\documentclass[11pt,a4paper]{amsart}

\usepackage{graphicx}
\newtheorem{theorem}{Theorem}
\newtheorem{definition}[theorem]{Definition}
\newtheorem{lemma}[theorem]{Lemma}
\newtheorem*{lemma*}{Lemma}
\newtheorem{proposition}[theorem]{Proposition}
\newtheorem{corollary}[theorem]{Corollary}

\newtheorem{remark}[theorem]{Remark}

\usepackage[utf8]{inputenc}
\usepackage[english]{babel}
\usepackage{latexsym}
\usepackage{amssymb}
\usepackage{amsmath}

\newcommand{\Z}{\mathbb{Z}}
\newcommand{\R}{\mathbb{R}}
\newcommand{\C}{\mathbb{C}}
\newcommand{\E}{\mathbb{E}}

\title[The Ginibre ensemble and the GFF]{On the logarithm of the characteristic polynomial of the Ginibre ensemble}

\author[C. Webb]{Christian Webb}
\address{Department of mathematics and systems analysis, Aalto University, PO Box 11000, 00076 Aalto, Finland}
\email{christian.webb@aalto.fi}
\date{\today}

\begin{document}

\begin{abstract}
We prove a slightly sharper version of a result of Rider and Vir\'ag who proved that after centering, the logarithm of the absolute value of the characteristic polynomial of the Ginibre ensemble converges in law to the Gaussian Free Field on the unit disk with free boundary conditions (and continued harmonically outside of it). Using their results on the linear statistics of the Ginibre ensemble, we prove a result that removes the ambiguity concerning the "constant part" or zero mode of the field.
\end{abstract}

\maketitle

\section{Introduction}

The main goal of this note is to make a small extension to a result of Rider and Vir\'ag in \cite{rv}. They prove that when one considers the (centered) logarithm of the absolute value of the characteristic polynomial of the Ginibre ensemble as a random distribution in a suitable space of generalized functions, it converges in the sense of finite dimensional distributions to a version of the 2-dimensional Gaussian Free Field, namely that with free boundary conditions on the unit disk and harmonically extended outside of the disk.

\vspace{0.3cm}

The extension we consider has to do with the fact that the space of distributions they consider can not distinguish between two distributions that differ by a constant - their space of distributions is the dual of the homogeneous Sobolev space $\dot{\mathcal{H}}^1(\C)$ (see e.g. \cite{logcor} for a survey on log-correlated fields some of which are defined only up to a constant). In addition to their space of distributions, their proof (which is based on studying linear statistics of the random matrices) does not distinguish between the logarithm of the absolute value of the characteristic polynomial and some (possibly random) constant translate of it. Our goal is to study the logarithm of the absolute value of the characteristic polynomial as an element of a space of distributions where one can remove this ambiguity in the "constant part" of the field. We also prove convergence in the weak sense which is stronger than the convergence in the sense of finite dimensional distributions in the infinite dimensional setting. We mainly make use of the results of \cite{rv} as well as fairly standard properties of the eigenfunctions of the Laplacian on the unit disk. Our approach is strongly motivated by a similar argument for the GUE that appears in \cite{fks}. Similar results might be possible for more general random matrix models such as those studied in \cite{ahm}.

\vspace{0.3cm}
  
Apart from it being slightly more satisfactory to remove this ambiguity in the constant part of the field, the motivation for studying the "constant part" of the field comes from recent developments in the study of log-correlated Gaussian fields. It has been discovered that their geometry is best studied through random measures which are obtained formally by exponentiating the field. The study of such measures is the theory of Gaussian Multiplicative chaos going back to Kahane \cite{kahane}. For a recent review, see e.g. \cite{gmc} and for a concise proof for the existence and uniqueness of multiplicative chaos measures see \cite{berestycki}. These measures play a role for example in the mathematical study of Liouville quantum gravity (see e.g. \cite{dskpz,dkrv}), random planar curves constructed by conformal welding \cite{ajks,sheffield}, quantum Loewner evolution \cite{qle}, and other random matrix models (see \cite{webb}, based on work in \cite{dik,ck}).

\vspace{0.3cm}

When exponentiating a log-correlated field, one regularizes it through a cut-off and studies how one must normalize the exponential of the regularized field to obtain a non-trivial limit when one removes the regularization. If the regularization of the field provided by the random matrix model possessed a very ill-behaved constant part (for example if the constant part of the field was a non-trivial non-gaussian random variable which exploded as the size of the matrix increased), this random variable would have to be dealt with separately in the exponentiation process. We will prove that there is in fact no problem with the constant part of the field, and this note can be seen as a small step in the problem of studying the relationship between the characteristic polynomial of the Ginibre ensemble and multiplicative chaos.
  
\vspace{0.3cm}

We start this note with a quick overview of some properties of the eigenfunctions of the Laplacian on the unit disk and related Sobolev spaces, we then go on to represent $z\mapsto \log |z-w|$ as an element of a Sobolev space. Next we'll recall the Ginibre ensemble and the main results of \cite{rv}. Finally we define the relevant log-correlated Gaussian Field and prove convergence of the recentered logarithm of the absolute value of the characteristic polynomial of the Ginibre ensemble to the log-correlated Gaussian field.

\vspace{0.3cm}

\bf Acknowledgements: \rm The author was supported by the Academy of Finland and wishes to thank N. Berestycki, Y. Fyodorov, A. Kupiainen, M. Nikula, R. Rhodes, E. Saksman, and V. Vargas for discussions related to this note.

\section{Eigenfuntions of the Laplacian on the unit disk and Sobolev spaces}

Let us first recall some basic facts about the eigenfunctions of the Laplacian on the unit disk, which we denote by $\mathbb{U}$.

\begin{definition}
Let $(e_{n,k})_{n\in \Z,k\in \Z_+}$ be the  eigenfunctions of the Laplacian on the unit disk with zero Dirichlet boundary conditions with unit $L^{2}(\mathbb{U})$ norm, namely 

\begin{equation}
e_{n,k}(r e^{i\phi})=C_{n,k}J_{|n|}(j_{|n|,k} r)e^{in\phi},
\end{equation}

\noindent where $J_n$ is a Bessel function of the first kind of order $n$, $j_{n,k}$ is the $k$th positive root of $J_n$, and

\begin{equation}
C_{n,k}=\frac{1}{\sqrt{\pi}}\frac{1}{J_{|n|+1}(j_{|n|,k})}.
\end{equation}

\end{definition}

\begin{remark}
The functions $e_{n,k}$ are orthonormal:

\begin{equation}
\int_{|z|<1}e_{n,k}(z)\overline{e_{m,l}(z)}d^2 z=\delta_{n,m}\delta_{k,l},
\end{equation}

\noindent the eigenvalue related to $e_{n,k}$ is $-j_{|n|,k}^2$, and one has the inequality (see e.g. \cite{besineq1} for the $n=0$ case and \cite{besineq2} for other values of $n$)

\begin{equation}\label{zeroineq}
j_{n,k}^2>n^2+\left(k-\frac{1}{4}\right)^2\pi^2.
\end{equation}

Moreover, as for integer $n$, $J_{-n}(x)=(-1)^n J_n(x)$, we have $j_{-n,k}=j_{n,k}$ so we can drop the absolute values from $j_{|n|,k}$.

\end{remark}

We will also make use of some rough bounds on $e_{n,k}$. While direct proofs surely exist as well, we refer to general results on eigenfunctions of the Laplacian on manifolds with boundary.

\begin{theorem}[\cite{eigest,eiggradest}]\label{th:eigest}
There exists a constant $C>0$ such that 

\begin{equation}
||e_{n,k}||_{L^{\infty}(\mathbb{U})}\leq C j_{n,k}
\end{equation}

\noindent and 

\begin{equation}
||\nabla e_{n,k}||_{L^{\infty}(\mathbb{U})}\leq C j_{n,k}^{3},
\end{equation}

\noindent where by $||\nabla e_{n,k}||_{L^{\infty}(\mathbb{U})}$ we mean $\sup_{z\in \mathbb{U}}|\nabla e_{n,k}(z)|$.

\end{theorem}

\vspace{0.3cm}

We now turn to the Sobolev spaces which will be the space of distributions where our fields will live.

\begin{definition}
For $s\in \R$, let 

\begin{equation}
\mathcal{H}^{s}=\left\lbrace f=\sum_{n\in \Z,k\in \Z_+}a_{n,k}e_{n,k}: \sum_{n\in \Z,k\in \Z_+}|a_{n,k}|^{2}j_{n,k}^{2s}<\infty\right\rbrace
\end{equation}

\noindent equipped with the inner product

\begin{equation}
\langle f,g\rangle=\sum_{n\in \Z,k\in\Z_+}a_{n,k}\overline{b_{n,k}}j_{n,k}^{2s},
\end{equation}

\noindent where $f=\sum_{n\in \Z,k\in \Z_+}a_{n,k}e_{n,k}$ and $g=\sum_{n\in \Z,k\in \Z_+}b_{n,k}e_{n,k}$.

\end{definition}

\begin{remark} For $s>0$, the notation $\mathcal{H}^{s}_0(\mathbb{U})$ might be more appropriate as this space can be viewed as a closure of the space of compactly supported smooth functions on $\mathbb{U}$ under a suitable norm, but for notational simplicity, we shall not carry the subscript or the reference to the domain $\mathbb{U}$. Also we shall not discuss other basis-independent definitions.

\end{remark}

\begin{remark} For any $s\in \R$ $\mathcal{H}^{s}$ is a separable Hilbert space with this inner product. We denote the corresponding norm by $||\cdot ||_s$. Moreover, for $s\geq 0$, the elements of $\mathcal{H}^{s}$ can be realized as a subspace of $L^{2}(\mathbb{U})$. $\mathcal{H}^{-s}$ on the other hand is the topological dual of $\mathcal{H}^{s}$ and it can be understood as a space of generalized functions acting on functions in $\mathcal{H}^{s}$: for $\phi=\sum_{n\in \Z,k\in \Z_+}\alpha_{n,k}e_{n,k}\in\mathcal{H}^{-s}$ and $f=\sum_{n\in \Z,k\in \Z_+}a_{n,k}e_{n,k}\in \mathcal{H}^{s}$, the action of $\phi$ on $f$ (which we want to formally understand as $\int_{\mathbb{U}}\phi(z)f(z)d^{2}z$) is 

\begin{equation}
\phi(f)=\sum_{n\in \Z,k\in \Z_+}\alpha_{n,k}a_{-n,k}.
\end{equation}
\end{remark}

We now turn to representing polynomials as elements in $\mathcal{H}^{-s}$ (one could of course understand them as elements in a smaller space, but this is the relevant representation for us).

\begin{lemma}\label{le:powerexp}
For $n\geq 0$,

\begin{equation}
z^{n}=\sum_{k=1}^{\infty} \frac{2\sqrt{\pi}}{j_{n,k}}e_{n,k}(z)
\end{equation}

\noindent as an element of $\mathcal{H}^{-s}$ for any $s>0$.

\end{lemma}

\begin{proof}
This is simply an issue of calculating the Fourier-Bessel series of the function $r\mapsto r^{n}$. Making use of the identity $\frac{d}{dx}(x^{n+1}J_{n+1}(x))=x^{n+1}J_n(x)$, we have for $n\geq 0$

\begin{align}
\notag\int_0^{1}r^{n+1}J_n(j_{n,k}r)dr &=\frac{1}{j_{n,k}^{2+n}}\int_0^{j_{n,k}}r^{n+1}J_n(r)dr\\
&=\frac{1}{j_{n,k}^{2+n}}\int_0^{j_{n,k}}\frac{d}{dr}(r^{n+1}J_{n+1}(r))dr\\
&=\frac{1}{j_{n,k}}J_{n+1}(j_{n,k})\notag 
\end{align}

\noindent implying that as elements in $L^{2}((0,1),rdr)$

\begin{equation}
r^{n}=\sum_{k=1}^{\infty}\frac{2}{j_{n,k} J_{n+1}(j_{n,k})}J_{n}(j_{n,k}r),
\end{equation}

\noindent which in turn implies that as elements in $L^{2}(\mathbb{U},d^{2}z)$

\begin{equation}
z^{n}=\sum_{k=1}^{\infty}\frac{2\sqrt{\pi}}{j_{n,k}}e_{n,k}(z).
\end{equation}

This then implies (say by dominated convergence arguments) that for any $f=\sum_{n\in \Z,k\in \Z_+}a_{n,k}e_{n,k}\in\mathcal{H}^{s}$ for some $s>0$,

\begin{equation}
\int_{|z|<1}z^{n}f(z)d^{2}z=\sum_{k=1}^{\infty}\frac{2\sqrt{\pi}}{j_{n,k}}a_{-n,k}
\end{equation}

\noindent which allows us to identify the mapping $z\mapsto z^{n}$ with $\sum_{k=1}^{\infty}\frac{2\sqrt{\pi}}{j_{n,k}}e_{n,k}$ when considered as elements of $\mathcal{H}^{-s}$ for any $s>0$.

\end{proof}

\begin{remark} In particular, we see that constants are elements of $\mathcal{H}^{-s}$ for $s>0$ and we can distinguish between elements differing by constants.
\end{remark}

\section{Representing $z\mapsto \log |z-w|$ as an element of $\mathcal{H}^{-s}$}

We wish to express the logarithm of the absolute value of the characteristic polynomial as an element of $\mathcal{H}^{-s}$ so let us begin by considering $z\mapsto \log |z-w|$ as an element of $\mathcal{H}^{-s}$ for any $w\in \C$. Before doing this, let us recall how $\log|z-w|$ is related to the Green's function of the Laplacian on $\mathbb{U}$ with zero Dirichlet boundary conditions. We also point out a simple expansion for $z\in \mathbb{U}$, $w\notin\mathbb{U}$.

\begin{lemma}\label{le:logrep}
For $z,w\in\mathbb{U}$, 

\begin{equation}
\log |z-w|=-2\pi \sum_{n\in \Z,k\in \Z_+}\frac{1}{j_{n,k}^{2}}e_{n,k}(z)e_{-n,k}(w)-\mathrm{Re}\sum_{n=1}^{\infty}\frac{1}{n}z^{n}\overline{w}^{n}.
\end{equation}

For $z\in \mathbb{U}$, $w\notin\mathbb{U}$

\begin{equation}
\log |z-w|=\log |w|-\mathrm{Re}\sum_{n=1}^{\infty}\frac{1}{n}\left(\frac{z}{w}\right)^{n}.
\end{equation}

\end{lemma}

\begin{proof}
The first part of the first sum is simply the Green's function of the Laplacian on $\mathbb{U}$ with zero Dirichlet boundary conditions, i.e. for 

\begin{equation}
G_D(z,w)=-\sum_{n\in \Z,k\in \Z_+}\frac{1}{j_{n,k}^{2}}e_{n,k}(z)e_{-n,k}(w)
\end{equation}

\noindent one has for $|w|<1$, $\Delta_z G_D(z,w)=\delta(z,w)$ and $G_D(z,w)\to 0$ as $z\to \partial\mathbb{U}$ from $\mathbb{U}$. On the other hand, as $\Delta_z \frac{1}{2\pi}\log |z-w|=\delta(z-w)$, we see that 

\begin{equation}
\frac{1}{2\pi}\log |z-w|-G_D(z,w)
\end{equation}

\noindent is a harmonic function of $z$ in $\mathbb{U}$ and for $|z|=1$ it equals $\frac{1}{2\pi}\log |z-w|=\frac{1}{2\pi}\log |1-\overline{z}w|$. Thus extending this harmonically gives

\begin{equation}
\log |z-w|-2\pi G_D(z,w)=\log |1-\overline{z}w|=-\mathrm{Re}\sum_{n=1}^{\infty}\frac{1}{n}z^{n}\overline{w}^{n}.
\end{equation}

This is the first claim. For $|w|\geq 1$ and $|z|<1$, one has 

\begin{equation}
\log |z-w|=\log|w|+\log \left|1-\frac{z}{w}\right|=\log |w|-\mathrm{Re}\sum_{n=1}^{\infty}\frac{1}{n}\left(\frac{z}{w}\right)^{n}.
\end{equation}

\end{proof}

Let us now show that using this expansion, we can represent $z\mapsto \log |z-w|$ as an element of $\mathcal{H}^{-s}$.

\begin{lemma}\label{le:coord}
For $s>0$, the mapping $z\mapsto \log |z-w|$ is an element of $\mathcal{H}^{-s}$ and equals $\sum_{n\in \Z,k\in \Z_+}\alpha_{n,k}(w)e_{n,k}$, where for $|w|<1$,

\begin{equation}
\alpha_{n,k}(w)=-\frac{2\pi}{j_{n,k}^{2}}e_{-n,k}(w)-\frac{\sqrt{\pi}(\mathbf{1}(n>0)\overline{w}^{n}+\mathbf{1}(n<0)w^{-n})}{j_{n,k}|n|}
\end{equation}

\noindent and for $|w|\geq 1$

\begin{equation}
\alpha_{n,k}(w)=\frac{2\sqrt{\pi}}{j_{n,k}}\left(\delta_{n,0}\log |w|-\mathbf{1}(n>0)\frac{1}{2nw^{n}}-\mathbf{1}(n<0)\frac{1}{2|n|\overline{w}^{|n|}}\right).
\end{equation}

\end{lemma}

\begin{proof}
As in both cases $z\mapsto \log |z-w|\in L^{2}(\mathbb{U},d^{2}z)$. This again amounts to calculating the Fourier-Bessel coefficients of the functions. For $|w|<1$, we have by the previous lemma

\begin{align}
&\int_{|z|<1}\log |z-w|e_{-n,k}(z)d^{2}z\notag\\
&=2\pi \int_{|z|<1} G_D(z,w)e_{-n,k}(z)d^{2}z-\int_{|z|<1}\sum_{m=1}^{\infty}\frac{1}{m}\mathrm{Re}(z^{m}\overline{w}^{m})e_{-n,k}(z)d^{2}z\notag\\
&=\notag -\frac{2\pi}{j_{n,k}^{2}}e_{-n,k}(w)-\frac{\sqrt{\pi}(\mathbf{1}(n>0)\overline{w}^{n}+\mathbf{1}(n<0)w^{-n})}{j_{n,k}|n|}.
\end{align}

Again the previous lemma implies that for $|w|\geq 1$, 

\begin{align}
\int_{|z|<1}&\log |z-w|e_{-n,k}(z)d^{2}z\notag\\
&=\log |w|\int_{|z|<1}e_{-n,k}(z)d^{2}z-\int_{|z|<1}\sum_{m=1}^{\infty}\frac{1}{m}\mathrm{Re}\left(\frac{z^{m}}{w^{m}}\right)e_{-n,k}(z)d^{2}z\\
&=\notag \frac{2\sqrt{\pi}}{j_{n,k}}\left(\delta_{n,0}\log |w|-\mathbf{1}(n>0)\frac{1}{2nw^{n}}-\mathbf{1}(n<0)\frac{1}{2|n|\overline{w}^{|n|}}\right).
\end{align}

\end{proof}

\section{The Ginibre ensemble and the results of \cite{rv}}
In this section we recall the Ginibre ensemble, the joint distribution of its eigenvalues and some results from \cite{rv} concerning the linear statistics of the Ginibre ensemble.

\begin{definition}[Ginibre ensemble]
We call a random ($\C^{N\times N}$-valued) variable $G_N$ a Ginibre random matrix or an element of the Ginibre ensemble if its entries are of the form $(G_N)_{i,j}=\frac{1}{\sqrt{N}}Z_{i,j}$, where $Z_{i,j}$ are i.i.d. standard complex Gaussians. We will denote the eigenvalues of $G_N$ by $(z_1,...,z_N)$.
\end{definition}

\begin{remark} We will think of $G_N$ for different $N$ as living on the same probability space which is the one generated by i.i.d. standard complex Gaussians $(Z_{i,j})_{i,j=1}^{\infty}$.
\end{remark}

\vspace{0.3cm}

The law of the eigenvalues of $G_N$ was discovered by Ginibre \cite{gin}:

\begin{proposition}[Ginibre]
The distribution of the eigenvalues of $G_N$ is given by 

\begin{equation}
\mathbb{P}(dz_1,...,dz_N)=\frac{1}{Z_N}\prod_{i<j}|z_i-z_j|^{2}\prod_{k=1}^{N}\frac{N}{\pi}e^{-N|z_k|^{2}}d^{2}z_k,
\end{equation}

\noindent where $d^{2}z_k$ is the Lebesgue measure on $\C$ and 

\begin{equation}
Z_N=\frac{\prod_{k=1}^{N}k!}{N^{\frac{N(N-1)}{2}}}.
\end{equation}

\end{proposition}

The main result of \cite{rv} was the following.

\begin{theorem}[\cite{rv} Theorem 1.1]\label{th:linstat}
Let $f:\C\to \R$ possess continuous partial derivatives in a neighborhood of $\mathbb{U}=\lbrace z\in \C: |z|<1\rbrace$, and grow at most exponentially at infinity. Then as $N\to\infty$, the distribution of the random variable $\sum_{k=1}^{N}(f(z_k)-\E(f(z_k)))$ converges to a (centered) normal distribution with variance 

\begin{equation}
\frac{1}{4\pi}\int_{\mathbb{U}}|\nabla f|^{2}d^{2}z+\frac{1}{2}\sum_{k\in \Z}|k||\widehat{f}(k)|^{2},
\end{equation}

\noindent where $\widehat{f}(k)=\frac{1}{2\pi}\int_0^{2\pi}f(e^{i\theta})e^{-ik\theta}d\theta$ are the Fourier coefficients of $f$ restricted to the unit circle.
\end{theorem}

While it is a rather direct corollary to the above theorem, let us emphasize what this implies for the functions $\alpha_{n,k}$ defined in Lemma \ref{le:coord}.

\begin{corollary}\label{cor:linstat}
Let 

\begin{equation}
\gamma_{n,k}^{(N)}=\sum_{i=1}^{N}(\alpha_{n,k}(z_i)-\E(\alpha_{n,k}(z_i))).
\end{equation}

Then for any finite collection of $(n_i,k_i)$, with $n_i\geq 0$

\begin{equation}
(\gamma_{n_1,k_1}^{(N)},\gamma_{n_2,k_2}^{(N)},\cdots)
\end{equation}

\noindent converges jointly in law to a centered Gaussian random vector

\begin{equation}
(\gamma_{n_1,k_1},\gamma_{n_2,k_2},\cdots)
\end{equation}

\noindent whose entries are independent and the law of $\gamma_{0,k}$ equals that of 

\begin{equation}
\frac{\sqrt{\pi}}{j_{0,k}}A
\end{equation}

\noindent where $A$ is a standard Gaussian and for $n\geq 1$, the law of $\gamma_{n,k}$ equals that of 

\begin{equation}
\frac{\sqrt{\pi}}{j_{n,k}}\left(Z+\frac{1}{\sqrt{n}}W\right),
\end{equation}

\noindent where $Z$ and $W$ are i.i.d. standard complex Gaussians.

\end{corollary}

\begin{remark} The condition $n_i\geq 0$ is not a restriction as $\gamma_{-n,k}^{(N)}=\overline{\gamma_{n,k}^{(N)}}$.
\end{remark}

\begin{proof}

By Cram\'er-Wold, it is enough to consider arbitrary linear combinations.  Thus let $t_{n,k},s_{n,k}\in \R$ for all $n\geq 0$ and $k\in \Z_+$ such that $t_{n,k},s_{n,k}\neq 0$  only for finitely many $n$ and $k$. Let us write

\begin{equation}
f(w)=\sum_{n=0}^{\infty}\sum_{k=1}^{\infty}\left(t_{n,k}\mathrm{Re}(\alpha_{n,k}(w))+s_{n,k}\mathrm{Im}(\alpha_{n,k}(w))\right)
\end{equation}

\noindent and 

\begin{equation}
\Gamma_N=\sum_{i=1}^{N}(f(w_i)-\E(f(w_i))).
\end{equation}

To apply the previous theorem, we need to check the regularity of $f$. It is clear that $\alpha_{n,k}$ are smooth in the unit disk and outside of it. To check smoothness across the boundary, we note that in polar coordinates, $w=re^{i\theta}$, we have for $r<1$

\begin{align}
\alpha_{n,k}(re^{i\theta})&=-\frac{2\sqrt{\pi}}{j_{n,k}^{2} J_{|n|+1}(j_{n,k})}J_{|n|}(j_{n,k}r)e^{-in\theta}-\mathbf{1}(n\neq 0)\frac{\sqrt{\pi}}{|n|j_{n,k}}r^{|n|}e^{-in\theta}.
\end{align}

\noindent and for $r\geq 1$,

\begin{equation}
\alpha_{n,k}(re^{i\theta})=\frac{2\sqrt{\pi}}{j_{n,k}}\left(\delta_{n,0}\log r-\mathbf{1}(n\neq 0)\frac{1}{2|n|}r^{-|n|}e^{-in\theta}\right).
\end{equation}

From the identity 

\begin{equation}
\frac{d}{dx}\left(\frac{J_n(x)}{x^{n}}\right)=-\frac{J_{n+1}(x)}{x^{n}},
\end{equation}

\noindent one can check that

\begin{equation}
J_{|n|}'(j_{n,k})=-J_{|n|+1}(j_{n,k}).
\end{equation}

Thus the radial derivate of $\alpha_{n,k}$ is continuous across the unit disk. From this, one can check the continuity of the partial derivatives of $\alpha_{n,k}$ and thus $f$.

\vspace{0.3cm}

We now need to find the Fourier coefficients $f$ restricted to the unit circle and the Dirichlet energy of $f$ in the unit disk. For the Fourier coefficients, we note that

\begin{equation}
\alpha_{n,k}(e^{i\theta})=-\mathbf{1}(n\neq 0)\frac{\sqrt{\pi}}{j_{n,k}|n|}e^{-in\theta}
\end{equation}

\noindent so that 

\begin{equation}
f(e^{i\theta})=-\sum_{n=0}^{\infty}\sum_{k=1}^{\infty}\frac{\sqrt{\pi}}{j_{n,k}n}\left(t_{n,k}\cos n\theta-s_{n,k}\sin n\theta\right)
\end{equation}

\noindent and for $n\neq 0$, 

\begin{equation}
\widehat{f}(n)=-\sum_{k=1}^{\infty}\frac{\sqrt{\pi}}{2j_{n,k}|n|}\left(t_{|n|,k}+i(\mathbf{1}(n>0)-\mathbf{1}(n<0))s_{|n|,k}\right).
\end{equation}

For the Dirichlet energy, we note that integrating by parts, $e_{n,k}$ vanishing on the unit circle, and the $L^{2}(\mathbb{U},d^{2}z)$ orthonormality of $e_{n,k}$ imply that

\begin{align}
\notag\int_{\mathbb{U}}\nabla e_{n,k}(z)\cdot \nabla e_{m,l}(z)d^{2}z&=-\int_{\mathbb{U}}e_{n,k}(z)\Delta e_{m,l}(z)d^{2}z\\
&=j_{m,l}^{2}\delta_{n,-m}\delta_{k,l}.
\end{align}

With a similar argument ($z\mapsto \overline{z}^{m}$ and $z\mapsto z^{m}$ are harmonic in for $m\in\Z_+$ and $e_{n,k}$ vanishes on the boundary)

\begin{equation}
\int_{\mathbb{U}}\nabla e_{n,k}(z)\cdot \nabla (\overline{z}^{m})d^{2}z=\int_{\mathbb{U}}\nabla e_{n,k}(z)\cdot \nabla (z^{m})d^{2}z=0.
\end{equation}

We also have for $n,m\in \Z_+$

\begin{equation}
\int_{\mathbb{U}}\nabla (z^{n})\cdot \nabla (z^{m})d^{2}z=\int_{\mathbb{U}}\nabla (\overline{z}^{n})\cdot \nabla (\overline{z}^{m})d^{2}z=0.
\end{equation}

Finally integrating by parts (or a direct calculation) gives

\begin{align}
\int_{\mathbb{U}}\nabla (z^{n})\cdot \nabla(\overline{z}^{m})d^{2}z=\int_0^{2\pi}e^{in\phi} m e^{-im\phi}d\phi=2\pi m \delta_{m,n}.
\end{align}

Thus 

\begin{equation}
\int_{\mathbb{U}}\nabla \alpha_{n,k}(z)\cdot \nabla \alpha_{m,l}(z)d^{2}z=\delta_{m,-n}\left(\delta_{k,l}\frac{4\pi^{2}}{j_{n,k}^{2}}+\mathbf{1}(n\neq 0)\frac{2\pi^{2}}{|n|j_{n,k}j_{n,l}}\right).
\end{equation}

From the fact $\overline{\alpha_{n,k}(z)}=\alpha_{-n,k}(z)$ one has 

\begin{equation}
\int_{\mathbb{U}}\nabla \mathrm{Re}(\alpha_{n,k}(z))\cdot \nabla \mathrm{Im}(\alpha_{m,l}(z))d^{2}z=0,
\end{equation}

\begin{align}
\notag &\int_{\mathbb{U}}\nabla \mathrm{Re}(\alpha_{n,k}(z))\cdot \nabla \mathrm{Re}(\alpha_{m,l}(z))d^{2}z\\
&=\frac{1}{2}(\delta_{m,n}+\delta_{m,-n})\left(\delta_{k,l}\frac{4\pi^{2}}{j_{n,k}^{2}}+\mathbf{1}(n\neq 0)\frac{2\pi^{2}}{|n|j_{n,k}j_{n,l}}\right),
\end{align}

\noindent and

\begin{align}
\notag &\int_{\mathbb{U}}\nabla \mathrm{Im}(\alpha_{n,k}(z))\cdot \nabla \mathrm{Im}(\alpha_{m,l}(z))d^{2}z\\
&=\frac{1}{2}(\delta_{m,n}-\delta_{m,-n})\left(\delta_{k,l}\frac{4\pi^{2}}{j_{n,k}^{2}}+\mathbf{1}(n\neq 0)\frac{2\pi^{2}}{|n|j_{n,k}j_{n,l}}\right).
\end{align}

We conclude that 

\begin{align}
\notag\frac{1}{4\pi}\int_{\mathbb{U}}|\nabla f(z)|^{2}d^{2}z&=\frac{1}{8\pi}\sum_{n=0}^{\infty}\sum_{k,l=1}^{\infty}(t_{n,k}t_{n,l}(1+\delta_{n,0})+s_{n,k}s_{n,l}(1-\delta_{n,0}))\\
&\notag\times\left(\delta_{k,l}\frac{4\pi^{2}}{j_{n,k}^{2}}+\mathbf{1}(n\neq 0)\frac{2\pi^{2}}{nj_{n,k}j_{n,l}}\right)\\
&=\frac{\pi}{2}\sum_{n=0}^{\infty}\sum_{k=1}^{\infty}\frac{t_{n,k}^{2}(1+\delta_{n,0})+s_{n,k}^{2}(1-\delta_{n,0})}{j_{n,k}^{2}}\\
\notag &+\frac{\pi}{4}\sum_{n,k,l=1}^{\infty}\frac{t_{n,k}t_{n,l}+s_{n,k}s_{n,l}}{nj_{n,k}j_{n,l}}.
\end{align}

Thus as $N\to\infty$,  $\Gamma_N$ converges in law to a centered Gaussian random variable with covariance

\begin{align}
&\notag\frac{\pi}{2}\sum_{n=0}^{\infty}\sum_{k=1}^{\infty}\frac{t_{n,k}^{2}(1+\delta_{n,0})+s_{n,k}^{2}(1-\delta_{n,0})}{j_{n,k}^{2}}+\frac{\pi}{4}\sum_{n,k,l=1}^{\infty}\frac{t_{n,k}t_{n,l}+s_{n,k}s_{n,l}}{nj_{n,k}j_{n,l}}\\
&+\frac{1}{2}\sum_{n\in \Z\setminus \lbrace 0\rbrace}|n||\widehat{f}(n)|^{2}\\
&\notag=\frac{\pi}{2}\sum_{n=0}^{\infty}\sum_{k=1}^{\infty}\frac{t_{n,k}^{2}(1+\delta_{n,0})+s_{n,k}^{2}(1-\delta_{n,0})}{j_{n,k}^{2}}+\frac{\pi}{2}\sum_{n,k,l=1}^{\infty}\frac{t_{n,k}t_{n,l}+s_{n,k}s_{n,l}}{nj_{n,k}j_{n,l}}.
\end{align}

Consider sequences of random variables $(A_k)_{k=1}^{\infty}$, $(B_{n,k})_{n,k=1}^{\infty}$, $(C_{n,k})_{n,k=1}^{\infty}$, $(D_n)_{n=1}^{\infty}$ and $(E_n)_{n=1}^{\infty}$ where all of the appearing random variables are i.i.d. standard Gaussians. From the covariance formula, it is then immediate that the distribution of $\lim_{N\to\infty}\Gamma_N$ agrees with that of 

\begin{equation}
V=\sqrt{\pi}\sum_{k=1}^{\infty}\frac{t_{0,k}A_k}{j_{0,k}}+\sqrt{\frac{\pi}{2}}\sum_{n,k=1}^{\infty}\frac{t_{n,k}\left(B_{n,k}+\frac{1}{\sqrt{n}}D_n\right)+s_{n,k}\left(C_{n,k}+\frac{1}{\sqrt{n}}E_n\right)}{j_{n,k}}.
\end{equation}

If we then introduce $(Z_{n,k})_{n,k=1}^{\infty}$ i.i.d. standard complex Gaussians and $(W_n)_{n=1}^{\infty}$ i.i.d. standard complex Gaussians independent of the $Z$-variables, and define $Q_{n,k}=\frac{\sqrt{\pi}}{j_{n,k}}\left(Z_{n,k}+\frac{1}{\sqrt{n}}W_n\right)$, we see that the law of $V$ agrees with the law of 

\begin{equation}
\sum_{k=1}^{\infty}\frac{\sqrt{\pi}t_{0,k}}{j_{0,k}}A_k+\sum_{n,k=1}^{\infty}\left(t_{n,k}\mathrm{Re}(Q_{n,k})+s_{n,k}\mathrm{Im}(Q_{n,k})\right),
\end{equation}

\noindent which by Cram\'er-Wold is what was claimed.

\end{proof}

\section{The Gaussian field}

Plugging the corollary of the main result of \cite{rv} that we proved in the previous section into the expansion $\log|z-w|=\sum_{n,k}\alpha_{n,k}(w)e_{n,k}(z)$ motivates defining the following object.

\begin{definition}
Let $(A_k)_{k=1}^{\infty}$ be i.i.d. standard Gaussians, let $(Z_{n,k})_{n,k=1}^{\infty}$ be i.i.d. standard complex Gaussians independent of the $A$ vairables, and let $(W_n)_{n=1}^{\infty}$ be i.i.d. standard complex Gaussians independent of the $A$ and $Z$ variables. Denote by $h$ the formal sum

\begin{equation}
h=\sqrt{\pi}\sum_{k=1}^{\infty}\frac{A_k}{j_{0,k}}e_{0,k}+2\sqrt{\pi}\mathrm{Re}\left(\sum_{n,k=1}^{\infty}\frac{1}{j_{n,k}}\left(Z_{n,k}+\frac{1}{\sqrt{n}}W_n\right)e_{n,k}\right)
\end{equation}
\end{definition}

This is not simply a formal sum but can be realized as a $\mathcal{H}^{-s}$-valued random variable:

\begin{lemma}
For any $s>0$, the series defining $h$ converges almost surely in $\mathcal{H}^{-s}$.
\end{lemma}

\begin{proof}
We have 

\begin{equation}
||h||_{-s}^{2}=\pi \sum_{k=1}^{\infty}\frac{A_k^{2}}{j_{0,k}^{2+2s}}+2\pi \sum_{n,k=1}^{\infty}\frac{\left|Z_{n,k}+\frac{1}{\sqrt{n}}W_n\right|^{2}}{j_{n,k}^{2+2s}}
\end{equation}

\noindent so that 

\begin{equation}
\E\left(||h||_{-s}^{2}\right)=\pi \sum_{k=1}^{\infty}\frac{1}{j_{0,k}^{2+2s}}+2\pi \sum_{n,k=1}^{\infty}\frac{\left(1+\frac{1}{n}\right)}{j_{n,k}^{2+2s}}.
\end{equation}

Recalling the estimate \eqref{zeroineq}, we see that there are some positive constants $C_1,C_2$ such that 

\begin{equation}
\E\left(||h||_{-s}^{2}\right)\leq C_1 \sum_{k=1}^{\infty }k^{-2-2s}+C_2\sum_{n,k=1}^{\infty}(n^{2}+k^{2})^{-1-s}
\end{equation}

\noindent which is finite for all $s>0$. This implies that $||h||_{-s}<\infty$ almost surely so $h\in\mathcal{H}^{-s}$ almost surely.
\end{proof}

\begin{remark} Formally, one finds that the covariance kernel of the field is (for $|z|,|w|<1$)

\begin{align}
\notag\E(h(z)h(w))&=\pi\sum_{k=1}^{\infty}\frac{1}{j_{0,k}^{2}}e_{0,k}(z)e_{0,k}(w)\\
\notag &+\pi\sum_{n,k=1}^{\infty}\frac{1}{j_{n,k}^{2}}\left(e_{n,k}(z)e_{-n,k}(w)+e_{-n,k}(z)e_{n,k}(w)\right)\\
&+\pi\sum_{n,k,l=1}^{\infty}\frac{1}{nj_{n,k}j_{n,l}}(e_{n,k}(z)e_{-n,l}(w)+e_{-n,k}(z)e_{n,l}(w))\\
&=\notag  -\pi G_D(z,w)+\frac{1}{4}\sum_{n=1}^{\infty}\frac{1}{n}\left(z^{n}\overline{w}^{n}+\overline{z}^{n}w^{n}\right)\\
\notag &=-\pi G_D(z,w)-\frac{1}{2}\log |1-z\overline{w}|\\
&=-\frac{1}{2}\log |z-w|.\notag
\end{align}

Here we used Lemma \ref{le:powerexp} and Lemma \ref{le:logrep}. Making this fact precise would be simple, but we skip it.

\vspace{0.3cm}

As the covariance of the field is (proportional to) $-\log |z-w|$ which is the covariance of the whole plane Gaussian Free Field, we interpret this as the field being the restriction of the GFF to the unit disk, or the free field with free boundary conditions on the unit disk. Of course the whole plane field is defined only up to a constant, so in a sense we fix the constant in our result.

\vspace{0.3cm}

We also point out that this field is ($\frac{1}{2}$ times) a sum of a Gaussian Free Field with zero Dirichlet boundary conditions on the disk and an independent harmonic extension of a Gaussian field defined on the unit disk through the random Fourier series $\mathrm{Re}\sum_{n=1}^{\infty}\frac{1}{\sqrt{n}}e^{in\theta}Z_n$, where $Z_n$ are i.i.d. standard complex Gaussians.
\end{remark}

\section{Convergence of the logarithm of the absolute value of the characteristic polynomial}

The basic idea for proving that the logarithm of the absolute value of the centered characteristic polynomial converges to $h$ is to control the behavior of $\gamma_{n,k}^{(N)}$ as $n$ or $k$ grows with $N$ and if $\gamma_{n,k}^{(N)}$ does not grow too fast, one is able to prove convergence in some $\mathcal{H}^{-s}$ for the suitable $s$ depending on the growth rate of $\gamma_{n,k}^{(N)}$. We suspect that actually $\gamma_{n,k}^{(N)}$ decays with $n$ and $k$ and this decay might actually be fast enough to prove convergence in any $\mathcal{H}^{-s}$ for $s>0$ (see the remarks at the end of the section), though our focus is on determining the constant part of the field instead of the precise roughness of the field.

\vspace{0.3cm}

Before going into the actual statement of the result and its proof, let us prove the variance bound for $\gamma_{n,k}^{(N)}$ we shall make use of. We suspect that this is a very rough estimate and one could do much better, but this is extremely simple and sufficient for us.

\begin{lemma}\label{le:varbound}
There exists a constant $C>0$ such that for $n\in \Z$, $k\in \Z_+$ one has (recall the notation of Corollary \ref{cor:linstat})
\begin{equation}
\E\left(\left|\gamma_{n,k}^{(N)}\right|^{2}\right)\leq C j_{n,k}^{2}.
\end{equation}

\end{lemma}

\begin{proof}
Our starting point is the following formula that follows from the determinantal structure of the distribution of the eigenvalues. For any (measurable) $f:\C\to\C$,

\begin{align}
\notag &\E\left(\left|\sum_{i=1}^{N}\left(f(z_i)-\E(f(z_i))\right)\right|^{2}\right)\\
&=\frac{1}{2}\left(\frac{N}{\pi}\right)^{2}\int_{\C\times \C}|f(z)-f(w)|^{2}\left|\sum_{k=0}^{N-1}\frac{(Nz\overline{w})^{k}}{k!}\right|^{2}e^{-N|z|^{2}-N|w|^{2}}d^{2}zd^{2}w.
\end{align}

Now if we assume that $f$ is smooth and $|\nabla f|$ is bounded, we find

\begin{align}
\notag &\E\left(\left|\sum_{i=1}^{N}\left(f(z_i)-\E(f(z_i))\right)\right|^{2}\right)\\
&\leq \frac{||  \nabla f  ||_{\infty}^{2}}{2}\left(\frac{N}{\pi}\right)^{2}\int_{\C\times \C}|z-w|^{2}\left|\sum_{k=0}^{N-1}\frac{(Nz\overline{w})^{k}}{k!}\right|^{2}e^{-N|z|^{2}-N|w|^{2}}d^{2}zd^{2}w.
\end{align}

The integral here is simple enough to calculate exactly so let us do that 

\begin{align*}
&\frac{1}{2}\left(\frac{N}{\pi}\right)^{2}\int_{\C\times \C}|z-w|^{2}\left|\sum_{k=0}^{N-1}\frac{(Nz\overline{w})^{k}}{k!}\right|^{2}e^{-N|z|^{2}-N|w|^{2}}d^{2}zd^{2}w\\
&=\frac{1}{2}\left(\frac{N}{\pi}\right)^{2}\sum_{k,l=0}^{N-1}\frac{N^{k+l}}{k!l!}\int_{\C\times\C}(|z|^{2}-\overline{z}w-\overline{w}z+|w|^{2})\\
&\times z^{k}\overline{z}^{l}w^{l}\overline{w}^{k}e^{-N|z|^{2}-N|w|^{2}}d^{2}zd^{2}w\\
&=\left(\frac{N}{\pi}\right)^{2}\Bigg(\sum_{k=0}^{N-1}\frac{N^{2k}}{k!^{2}}\int_\C |z|^{2k+2}e^{-N|z|^{2}}d^{2}z\int_\C |w|^{2k}e^{-N|w|^{2}}d^{2}w\\
&-\sum_{k=0}^{N-2}\frac{N^{2k+1}}{k!(k+1)!}\left(\int_\C |z|^{2k+2}e^{-N|z|^{2}}d^{2}z\right)^{2}\Bigg)\\
&=\left(\frac{N}{\pi}\right)^{2}\left(\sum_{k=0}^{N-1}\pi^{2}(k+1)N^{-3}-\sum_{k=0}^{N-2}\pi^{2}(k+1)N^{-3}\right)\\
&=1.
\end{align*}

We used here the simple fact that for $m\in \Z_+$,

\begin{equation}
\int_\C |z|^{2m}e^{-N|z|^{2}}d^{2}z=\frac{\pi m!}{N^{m+1}}.
\end{equation}

We thus conclude that 

\begin{equation}
\E\left(\left|\gamma_{n,k}^{(N)}\right|^{2}\right)\leq ||\nabla \alpha_{n,k}||_{\infty}^{2}.
\end{equation}

From Lemma \ref{le:coord}, we see that there exist positive constants $C_1$ and $C_2$ such that 

\begin{equation}
\sup_{z\in \mathbb{U}}|\nabla \alpha_{n,k}(z)|\leq \frac{C_1}{j_{n,k}^{2}}||\nabla e_{-n,k}(z)||_{L^{\infty}(\mathbb{U})}+C_2\frac{1}{j_{n,k}}
\end{equation}

\noindent while outside of the disk we have a positive constant $C_3$ such that 

\begin{equation}
\sup_{z\notin \mathbb{U}}|\nabla \alpha_{n,k}(z)|\leq \frac{C_3}{j_{n,k}}.
\end{equation}

Making use of Theorem \ref{th:eigest}, we find that for some positive constant $C$

\begin{equation}
||\nabla \alpha_{n,k}||_\infty\leq C j_{n,k},
\end{equation}

\noindent which yields the claim.

\end{proof}

Let us now introduce notation for the centered logarithm of the absolute value of the characteristic polynomial.

\begin{definition}
For $z\in \mathbb{U}$ and $N\in \Z_+$, write 

\begin{equation}
h_N(z)=\sum_{i=1}^{N}\left(\log |z-z_i|-\E(\log |z-z_i|)\right).
\end{equation}

\end{definition}

We can now prove the convergence of $h_N$ to $h$ in the sense of finite dimensional distributions in the space $\mathcal{H}^{-s}$ for any $s>2$.  The argument is almost identical to that in \cite{fks}.

\begin{proposition}
Let $s>2$ and $k\in \Z_+$. For any elements $f_1,...,f_k\in \mathcal{H}^{s}$, $(h_N(f_1),...,h_N(f_k))$ converges in law to $(h(f_1),...,h(f_k))$.
\end{proposition}

\begin{proof}
By Cram\'er-Wold, it's enough to consider linear combinations and by linearity, this reduces to proving that for any $f\in \mathcal{H}^{s}$, $h_N(f)$ converges in law to $h(f)$. Let us write 

\begin{equation}
f=\sum_{n\in \Z,k\in \Z_+}a_{n,k}e_{n,k}
\end{equation}

\noindent so that 

\begin{equation}
h_N(f)=\sum_{n\in \Z,k\in \Z_+}\gamma_{n,k}^{(N)}a_{-n,k}.
\end{equation}

We now introduce a cut-off into the sums: let $M\in \Z_+$ and define

\begin{equation}
\epsilon_{N,M}=\sum_{\stackrel{n\in \Z,k\in \Z_+}{n^{2}+k^{2}>M}}\gamma_{n,k}^{(N)}a_{-n,k}.
\end{equation}

As Corollary \ref{cor:linstat} implies that for fixed $M$, as $N\to\infty$

\begin{equation}
h_N(f)-\epsilon_{N,M}\stackrel{d}{\to}\sum_{\stackrel{n\in \Z,k\in \Z_+}{n^{2}+k^{2}\leq M}}\gamma_{n,k}a_{-n,k}
\end{equation}

\noindent and as we let $M\to\infty$, this in turn converges to $h(f)$. Thus by Slutsky's theorem, to prove that $h_N(f)\stackrel{d}{\to}h(f)$, it is enough to show that $\epsilon_{N,M}$ converges to zero in probability as we first let $N\to\infty$ and then $M\to \infty$. To see that this occurs, we note that by Cauchy-Schwarz (applied to the sum, not the integral) and our variance bound (\ref{le:varbound})

\begin{align}
\notag\E\left(\left|\epsilon_{N,M}\right|^{2}\right)&\leq \sum_{\stackrel{n\in \Z,k\in \Z_+}{n^{2}+k^{2}> M}}\E\left(\left|\gamma_{n,k}^{(N)}\right|^{2}\right)j_{n,k}^{-2s} \sum_{\stackrel{n\in \Z,k\in \Z_+}{n^{2}+k^{2}>M}}|a_{-n,k}|^{2}j_{n,k}^{2s}\\
&\leq C||f||_s\sum_{\stackrel{n\in \Z,k\in \Z_+}{n^{2}+k^{2}> M}}j_{n,k}^{2-2s}.
\end{align}

As $s>2$, \eqref{zeroineq} implies that this last series converges for any $M\in \Z_+$ so it tends to zero as $M\to\infty$. This estimate was uniform in $N$, so we see that $\epsilon_{N,M}$ tends to zero in probability if we first let $N\to\infty$ and then $M\to\infty$. Thus $h_N(f)$ converges in law to $h(f)$.

\end{proof}

\begin{remark} It is clear from the proof that improving the bound $\E(|\gamma_{n,k}|^{2})\leq Cj_{n,k}^{2}$ to say something of the form $j_{n,k}^{2\alpha}$ will improve the lower bound of $s$ to $s>1+\alpha$. So in particular if there were decay of the form $j_{n,k}^{-2}$, one would have convergence in $\mathcal{H}^{-s}$ for any $s>0$.
\end{remark}

\vspace{0.3cm}

In an infinite dimensional space, weak convergence is a stronger form of convergence than convergence of finite dimensional distributions. A sufficient condition to strengthen convergence of finite dimensional distributions to weak convergence in the Hilbert case situation, is tightness (this follows essentially from Prohorov's theorem - see e.g. \cite{merkle}). Let us now prove tightness and thus weak convergence. Again the argument is almost identical to that in \cite{fks}.

\begin{proposition}
For any $s>2$, $h_N$ converges weakly to $h$ in $\mathcal{H}^{-s}$.
\end{proposition}

\begin{proof}
As mentioned, as we know convergence of finite dimensional distributions, it is enough to prove tightness of the sequence $(h_N)_N$ in $\mathcal{H}^{-s}$. As in \cite{fks}, we make use of the fact (which one can prove again e.g. by a minor modification of the proof of Theorem 8.3 in \cite{kress}) that for $2<s'<s$, the ball 

\begin{equation}
K_\epsilon=\left\lbrace \phi\in \mathcal{H}^{-s'}: ||\phi||_{-s'}\leq \frac{C}{\epsilon}\right\rbrace
\end{equation}

\noindent is compact in $\mathcal{H}^{-s}$. One then has by Chebyshev's inequality and our variance bound

\begin{align}
\notag \mathbb{P}(h_N\in K_\epsilon)&=1-\mathbb{P}\left(||h_N||_{-s'}>\frac{C}{\epsilon}\right)\\
\notag &\geq 1-\frac{\epsilon^{2}}{C^{2}}\E\left(||h_N||_{-s'}^{2}\right)\\
&=1-\frac{\epsilon^{2}}{C^{2}}\sum_{n\in \Z,k\in \Z_+}\E\left(\left|\gamma_{n,k}^{(N)}\right|^{2}\right)j_{n,k}^{-2s'}\\
\notag &\geq 1-C'\epsilon^{2}\sum_{n\in \Z,k\in \Z_+}j_{n,k}^{2-2s'}\\
&\notag \geq 1-\widetilde{C}(s') \epsilon^{2},
\end{align}

\noindent where we again used the fact that for $s'>2$, $\sum_{n\in \Z,k\in \Z_+}j_{n,k}^{2-2s'}$ converges. Thus we have tightness as well as weak convergence.

\end{proof}

\begin{remark} As we already pointed out, the space where the proof guarantees convergence in depends on the bounds one has for $\E(|\gamma_{n,k}^{(N)}|^{2})$. We expect that our bound of the form $C j_{n,k}^{2}$ is nowhere near optimal. To motivate this, consider the situation where $|n|\geq N$.

\vspace{0.3cm}

Another way to write the variance of $\gamma_{n,k}^{(N)}$ is 

\begin{align}
\notag &\E\left(\left|\gamma_{n,k}\right|^{2}\right)=\frac{N}{\pi}\int_\C |\alpha_{n,k}(z)|^{2}\sum_{k=0}^{N-1}\frac{(N|z|^{2})^{k}}{k!}e^{-N|z|^{2}}d^{2}z\\
&-\left(\frac{N}{\pi}\right)^{2}\int_{\C\times \C}\alpha_{n,k}(z)\overline{\alpha_{n,k}(w)}\left|\sum_{k=0}^{N-1}\frac{(Nz\overline{w})^{k}}{k!}\right|^{2}e^{-N|z|^{2}-N|w|^{2}}d^{2}zd^{2}w.
\end{align}

Recall that if $z=re^{i\theta}$, then the angular dependence of $\alpha_{n,k}(z)$ is $e^{-in\theta}$. The angular part of $\left|\sum_{k=0}^{N-1}\frac{(Nz\overline{w})^{k}}{k!}\right|^{2}$ has only terms of the form $e^{im\theta}e^{im'\phi}$ ($\theta$ and $\phi$ being the phases of $z$ and $w$) with $-N+1\leq m,m'\leq N-1$ so we see that the double integral vanishes when $|n|\geq N$. We are thus left with estimating 

\begin{equation}
\frac{N}{\pi}\int_\C |\alpha_{n,k}(z)|^{2}\sum_{k=0}^{N-1}\frac{(N|z|^{2})^{k}}{k!}e^{-N|z|^{2}}d^{2}z.
\end{equation}

We split this into an integral over $|z|<1$ and $|z|\geq 1$. In the $|z|<1$ case we note that $\sum_{k=0}^{N-1}\frac{(N|z|^{2})^{k}}{k!}e^{-N|z|^{2}}\leq 1$ and 

\begin{align}
\notag\int_{|z|<1}|\alpha_{n,k}(z)|^{2}d^{2 }z&=\frac{4\pi^{2}}{j_{n,k}^{4}}\int_{|z|<1}|e_{n,k}(z)|^{2}d^{2}z+\frac{4\pi\sqrt{\pi}}{j_{n,k}^{3}|n|}\int_{|z|<1}e_{|n|,k}(z)\overline{z}^{|n|}d^{2}z\\
\notag &+\frac{\pi}{j_{n,k}^{2}|n|^{2}}\int_{|z|<1}|z|^{2|n|}d^{2}z\\
&\leq \frac{C}{j_{n,k}^{2}|n|^{2}}.\notag
\end{align}

For $|z|>1$, we note that $|\alpha_{n,k}(z)|\leq \frac{C}{j_{n,k}|n|}$ so 

\begin{align}
&\notag\frac{N}{\pi}\int_{|z|>1} |\alpha_{n,k}(z)|^{2}\sum_{k=0}^{N-1}\frac{(N|z|^{2})^{k}}{k!}e^{-N|z|^{2}}d^{2}z\\
&\leq \frac{C N}{j_{n,k}^{2}|n|^{2}}\sum_{k=0}^{N-1}\frac{N^{k}}{k!}\int_{\C}|z|^{2k}e^{-N|z|^{2}}d^{2}z\\
&=\frac{C' N}{j_{n,k}^{2}|n|^{2}}\notag.
\end{align}

We conclude that for $|n|\geq N$, one has

\begin{equation}
\mathbb{E}\left(\left|\gamma_{n,k}^{(N)}\right|^{2}\right)\leq C\frac{1}{|n|j_{n,k}^{2}}
\end{equation}

\noindent which suggests that in general one has decay instead of growth and one could probably improve our results in terms of the roughness of the space of distributions.
\end{remark}

\end{document}